\documentclass[10pt,a4paper]{amsart}

\theoremstyle{plain}
\newtheorem{theorem}{Theorem}[section]
\newtheorem{lemma}[theorem]{Lemma}
\newtheorem{proposition}[theorem]{Proposition}
\newtheorem{corollary}[theorem]{Corollary}
\theoremstyle{definition}

\theoremstyle{remark}
\newtheorem{remark}[theorem]{Remark}

\newtheorem{notations}[theorem]{Notations}

\newcommand{\ic}{\ensuremath{\mathcal{I}}}
\newcommand{\oc}{\ensuremath{\mathcal{O}}}

\newcommand{\Ps}{\mathbb{P}}

\def\bin #1#2 {\left( \matrix { #1 \cr #2 \cr } \right) }

\begin{document}

\title[Noether-Lefschetz Theory  with base locus]
{Noether-Lefschetz Theory with base locus}

\author{Vincenzo Di Gennaro }
\address{Universit\`a di Roma \lq\lq Tor Vergata\rq\rq, Dipartimento di Matematica,
Via della Ricerca Scientifica, 00133 Roma, Italy.}
\email{digennar@axp.mat.uniroma2.it}

\author{Davide Franco }
\address{Universit\`a di Napoli
\lq\lq Federico II\rq\rq, Dipartimento di Matematica e
Applicazioni \lq\lq R. Caccioppoli\rq\rq, P.le Tecchio 80, 80125
Napoli, Italy.} \email{davide.franco@unina.it}

\abstract Let $Z$ be a closed subscheme of a  smooth complex
projective  variety $Y\subseteq \Ps^N$, with $\dim\,Y=2r+1\geq 3$.
We describe the intermediate N\'eron-Severi  group  (i.e.  the
image of the cycle map $A_r(X)\to H_{2r}(X;\mathbb{Z})$) of a
general smooth hypersurface $X\subset Y$ of sufficiently large
degree containing $Z$.



\endabstract
\maketitle

\section{Introduction}

The classical Noether-Lefschetz Theorem implies that the
N\'eron-Severi group of a (very) general space surface $X\subset
\mathbb{P}^3=\mathbb{P}^3(\mathbb C)$, with degree $\geq 4$, is
generated by the hyperplane class. The proof  rests on two main
ingredients: a monodromy argument, showing that any class in the
N\'eron-Severi group of $X$ can be lifted to
$H^2(\mathbb{P}^3;\mathbb{Q})$ as a rational class; Lefschetz
Hyperplane Theorem, saying that the map
$H^2(\mathbb{P}^3;\mathbb{Z})\to H^2(X;\mathbb{Z})$ is injective
with torsion-free cokernel.

What can be said, in general, about the $i$-th N\'eron-Severi
group $NS_i(X;\mathbb Z)$ (i.e.  the image of the cycle map
$A_i(X)\to H_{2i}(X;\mathbb{Z})\cong H^{2\dim
X-2i}(X;\mathbb{Z})$, \cite{Fulton}, $\S 19.1$), for a general
hypersurface $X$ of a smooth projective variety $Y$? As far as we know
the more general result in this direction is due to Moishezon
(\cite{Moishezon}, Theorem 5.4, pag. 245), who provided a general
monodromy-type argument concerning the rational N\'eron-Severi
groups $NS_i(X;\mathbb Q)$ ($:= NS_i(X;\mathbb Z)\otimes_{\mathbb
Z} \mathbb{Q}$), from which one deduces that the natural map
$NS_{i+1}(Y;\mathbb Q)\to NS_i(X;\mathbb{Q})$ is surjective, as
soon as $h^{m,0}(X)>h^{m,0}(Y)$ ($m+1:=\dim Y$). Combining with
Lefschetz Hyperplane Theorem, saying that the map
$H^j(Y;\mathbb{Z})\to H^j(X;\mathbb{Z})$ is an isomorphism when
$j< m$ and injective with torsion-free cokernel when $j=m $, the
best one can obtain in general is:
\begin{equation}\label{general}
NS_i(X;\mathbb Z)\subseteq H^{2m-2i}(Y;\mathbb{Z}), \hskip1cm
i\geq \frac{m}{2}.
\end{equation}
So for instance one deduces:
\begin{equation}\label{generalalg}
NS_{i+1}(Y;\mathbb Z)\cong NS_i(X;\mathbb Z), \hskip1cm i\geq
\frac{m}{2}
\end{equation}
when the cohomology of $Y$ is algebraic. Let us say that the
unique hard case to prove is the intermediate one: $m=2i$.

By contrast, very few can be said in general about  the
N\'eron-Severi groups in  dimensions lower than the intermediate
one. For instance, it is still unknown whether  the degree of a
curve on a general threefold in $\mathbb{P}^4$ is a multiple of
the degree of the threefold (\cite{GH}, \cite{MKRR}). So there is
no hope to have a general result in small dimension.

The main purpose of our paper is to extend  (\ref{general}) and
(\ref{generalalg}) above to the general hypersurface containing a
given base locus (compare with \cite{Lopez}, \cite{BSN},
\cite{IJM2}). Let $Y\subseteq \Ps^N$ be a smooth complex
projective variety of dimension $m+1=2r+1\geq 3$, $Z$ be a closed
subscheme of $Y$, and $\delta$ be a positive integer such that
$\mathcal I_{Z,Y}(\delta)$ is generated by global sections. Assume
that for $d\gg 0$ the general divisor $X \in
|H^0(Y,\ic_{Z,Y}(d))|$ is smooth. This implies that $\dim\,Z\leq
r$ and that, for any $d\geq \delta$, there exists a smooth
hypersurface of degree $d$ containing $Z$ \cite{OS}. Our main
result concerns the intermediate N\'eron-Severi group
$NS_r(X;\mathbb Z)$, the higher cases being rather trivial. It
says, roughly speaking, that (\ref{general}) and
(\ref{generalalg}) above are corrected by a group which is freely
generated by the components of the base locus:

\begin{theorem}\label{maintheorem} With notations as above, let $X\in |H^0(Y,\mathcal O_Y(d))|$
be a  general divisor containing $Z$, with $d\geq \delta +1$.
Assume that the vanishing cohomology of $X$ is not of pure Hodge
type $(r,r)$. Denote by $H^m(X;\mathbb Z)_{Z}$ the subgroup of
$H^m(X;\mathbb Z)$ generated by the irreducible components
$Z_1,\dots,Z_{\rho}$ of $Z$ of dimension $r$. Then we have:

\smallskip
(a) $H^m(X;\mathbb Z)_{Z}$ is a free subgroup of $H^m(X;\mathbb
Z)$ of rank $\rho$;

\smallskip
(b) $NS_r(X;\mathbb Z)= \left[NS_r(X;\mathbb Z)\cap H^m(Y;\mathbb
Z)\right]\oplus H^m(X;\mathbb Z)_Z$;

\smallskip
(c) $NS_r(X;\mathbb Q)= NS_{r+1}(Y;\mathbb Q)\oplus H^m(X;\mathbb
Q)_Z$.
\end{theorem}

\noindent In the case $Z=\emptyset$, i.e. when $X$ is simply a
general hypersurface section of $Y$, this result easily follows
combining the quoted paper \cite{Moishezon} with Lefschetz
Hyperplane Theorem. It seems unknown whether the inclusion
$NS_{r+1}(Y;\mathbb Z)\subseteq NS_r(X;\mathbb Z)\cap
H^m(Y;\mathbb Z)$ is an equality, but there is some evidence
supporting this (\cite{Totaro}, Remark 1, p. 490).

The line of the proof of our Theorem is the following. Fix  smooth
divisors $G\in |H^0(Y,\mathcal O_Y(\delta))|$ and $X\in
|H^0(Y,\mathcal O_Y(d))|$ containing $Z$, and put $W:=G\cap X$ (by
(\cite{Vogel}, p. 133, Proposition 4.2.6. and proof) one knows
that $W$ has only isolated singularities). By an inductive method
(Theorem \ref{thmb}), in part already appearing in \cite{DGF} and
\cite{IJM2}, one reduces the proof to identify the subgroup
$I_W(\mathbb Z)\subseteq H^m(X;\mathbb Z)$ of the invariant
cocycles with respect to the monodromy representation on
$H^m(X_t;\mathbb Z)$ for the family of smooth divisors $X_t$ in
$|H^0(Y,\mathcal O_Y(d))|$ containing $W$. In the case of rational
coefficients we already know  that $I_W(\mathbb
Q)=H^m(Y;\mathbb{Q})+H^{m}(X_t;\mathbb{Q})_W$ \cite{DGF}, where
$H^{m}(X_t;\mathbb{Q})_W$ denotes the image of the push-forward
map $H_{m}(W;\mathbb{Q})\to H_{m}(X_t;\mathbb{Q})\cong
H^{m}(X_t;\mathbb{Q})$. However, unlike the case in which $Y$ is a
complete intersection (\cite{IJM2}, Theorem 2.3), in our general
setting classical Lefschetz Theory is not enough to deduce that
$H^{m}(X_t;\mathbb{Z})\slash
\left[H^m(Y;\mathbb{Z})+H^{m}(X_t;\mathbb{Z})_W\right]$ is torsion
free, hence that $I_W(\mathbb
Z)=H^m(Y;\mathbb{Z})+H^{m}(X_t;\mathbb{Z})_W$. We are able to
overcome this difficulty combining a more refined Lefschetz Theory
(see \cite{GMP3},  \cite{Hamm}, and  Theorem \ref{finv}, Lemma
\ref{semid}, and Lemma \ref{gmp} below), with a topological
description of the blowing-up $P:=Bl_W(Y)$ of $Y$ along $W$. This
decription relies on a sort of a decomposition theorem for the
integral homology of $P$ (Corollary \ref{ndecscoppdue}) for which,
even if many similar results already appear in the literarture
(\cite{Fulton}, \cite{Voisin}, \cite{Hartshorne2}, \cite{CG}), we
did not succeed in finding an appropriate reference.

\bigskip
\section{The group of invariants $I_W(\mathbb
Z)$.}

Let $Y\subseteq \Ps^N$  be a smooth complex projective variety in
$\Ps^N$, of odd dimension $m+1=2r+1\geq 3$. Fix integers $1\leq
k<d$, and smooth divisors $G\in |H^0(Y,\mathcal O_Y(k))|$ and
$X\in |H^0(Y,\mathcal O_Y(d))|$. Put
$$W:=G\cap X.$$
By (\cite{Vogel}, p. 133, Proposition 4.2.6. and proof) one knows
that $W$ has only isolated singularities. Let $I_W(\mathbb
Z)\subseteq H^m(X_t;\mathbb Z)$ be the group   of the invariant
cocycles with respect to the monodromy representation on
$H^m(X_t;\mathbb Z)$ for the family of smooth divisors $X_t$ in
$|H^0(Y,\mathcal O_Y(d))|$ containing $W$. The aim of this section
is to identify this group $I_W(\mathbb Z)$. In fact we are going
to prove that, at least under  suitable assumptions (unnecessary
in the case of rational coefficients), one has
\begin{equation}\label{ninv}
I_W(\mathbb Z)=H^m(Y;\mathbb Z)+H^m(X_t;\mathbb Z)_W
\end{equation}
(see Proposition \ref{semi} below), where we denote by
$H^m(X_t;\mathbb Z)_W$ the image of the push-forward map
$H_m(W;\mathbb Z)\to H_m(X_t;\mathbb Z)$ composed with Poincar\'e
duality $H_m(X_t;\mathbb Z)$ $\cong H^m(X_t;\mathbb Z)$, and  we
see $H^m(Y;\mathbb Z)$ contained in $H^m(X_t;\mathbb Z)$ via
pull-back thanks to Lefschetz Hyperplane  Theorem (one may give
similar definition with $\mathbb Q$ instead of $\mathbb Z$).
Equality (\ref{ninv}) relies on the study of the rational map
$Y\dasharrow \Ps$ $:=\Ps(H^0(Y,\ic_{W,Y}(d))^*)$ defined by the
linear system $|H^0(Y,\ic_{W,Y}(d))|$. For a geometric description
of this map, we refer to \cite{DGF1}, p. 755, and \cite{DGF}, p.
525. Here we simply recall the main properties. Next we will turn
to the topology of the blowing-up $P:=Bl_W(Y)$ of $Y$ along $W$.

\bigskip {\it{A geometric description  of the rational map
$Y\dasharrow \Ps$}}.
\smallskip

(a) Let $P$ be the blowing-up of $Y$ along $W$. For the strict
transforms $\widetilde G$ and ${\widetilde X}_t$ of $G$ and $X_t$
in $P$, one has $\widetilde G\cong G$, and ${\widetilde X}_t\cong
X_t$ when $G$ is not contained in $X_t$. By \cite{Fulton}, 4.4,
the rational map $Y\dasharrow \Ps :=\Ps(H^0(Y,\ic_{W,Y}(d))^*)$
defines a morphism $P\to \Ps$. Denote by $Q$ the image of this
morphism, i.e.:
$$
Q:=\Im(P\to \Ps)
$$
(compare with   \cite{Franchetta},  \cite{GH}).

\smallskip (b)
Set $E :=\Ps(\oc_{Y}(k)\oplus\oc_{Y}(d))$. The surjections
$\oc_{Y}(k)\oplus \oc_{Y}(d) \to \oc_{Y}(d)$ and $\oc_{Y}(k)\oplus
\oc_{Y}(d) \to \oc_{Y}(k)$ give rise to divisors $\Theta \cong
Y\subseteq E$ and $\Gamma\cong Y\subseteq E$, with $\Theta\cap
\Gamma =\emptyset$. The line bundle $\oc_{E}(\Theta )$ is base
point free and the corresponding morphism $E\to
\Ps(H^0(E,\oc_{E}(\Theta))^*)\cong \Ps$ sends $E$ to the cone $CY$
over $Y$ embedded via $|H^0(Y,\oc_{Y}(d-k))|$. This map contracts
$\Gamma $ to the vertex $v_\infty$ of the cone, and $\Theta $ to a
general hyperplane section  of $CY$. There is a natural closed
immersion $P\subset E$, and the trace of $|\Theta |$ on $P$,
giving the linear series spanned by the strict transforms
${\widetilde X}_t$, induces the map $P\to Q$. Hence we have a
natural commutative diagram:
$$
\begin{array}{ccccc}
 P&\hookrightarrow  & E  \\
\downarrow &\searrow & &\searrow  \\
Y & \dasharrow &Q&\hookrightarrow & CY\subseteq\Ps.\\
\end{array}
$$

\smallskip (c)
Moreover one has: $\Gamma \cap P = \widetilde G$; the map $P\to Q$
contracts $\widetilde G $ to $v_{\infty}\in Q$; $P \backslash
\widetilde G \cong Q \backslash \{v_{\infty}\}$; the hyperplane
sections $Q_t$ of $Q$, not containing the vertex, are isomorphic,
via $P\to Q $, to the corresponding divisors
$X_t\in|H^0(Y,\ic_{W,Y}(d))|$; the monodromy representation on
$H^m(X_t;\mathbb Z)$ for the family of smooth divisors $X_t$ in
$|H^0(Y,\mathcal O_Y(d))|$ containing $W$, identifies with the
monodromy representation on $H^m(Q_t;\mathbb Z)$ for the family of
smooth hyperplane sections $Q_t$ of $Q$; so as $W$, also $P$ and
$Q$ have only isolated singularities.

\bigskip
The following Theorem \ref{finv}  applies to $Q$ (with $Q=R$,
$m=n$, and $I_W(\mathbb Z)=I$). Recall that the inclusion
$X_t\cong Q_t \subset Q$ induces a Gysin map $H_{m+2}(Q;\mathbb
Z)\to H_m(X_t;\mathbb Z)$ (see \cite{Baum}, or \cite{Fulton}, p.
382, Example 19.2.1).

\begin{theorem}\label{finv} Let $R\subseteq\Ps^N$ be an irreducible, reduced, non-degenerate
projective variety of dimension $n+1\geq 2$, with isolated
singularities, and let $R_t$ be a general hyperplane section of
$R$. Denote by $i^{\star}_{k}:H_{k+2}(R;\mathbb Z)\to
H^{2n-k}(R_t;\mathbb Z)$ the map obtained composing the Gysin map
$H_{k+2}(R;\mathbb Z)\to H_{k}(R_t;\mathbb Z)$ with Poincar\'e
duality $H_{k}(R_t;\mathbb Z)\cong H^{2n-k}(R_t;\mathbb Z)$. Then
the following properties hold true.

\medskip
(a) For any  integer $n< k\leq 2n$ the map $i^{\star}_{k}$ is an
isomorphism, the map
 $i^{\star}_{n}$ is injective  with torsion-free cokernel, and $H_{n+2}(R;\mathbb
Z)\cong I$ via  $i^{\star}_{n}$, where  $ I\subseteq H^{n}(R_t;
\mathbb{Z})$ denotes the invariant subgroup given by the monodromy
action on the cohomology of $R_t$.

\medskip
(b) For any  even integer $n< k=2i\leq 2n$ the map
$i^{\star}_{k}\otimes_{ \mathbb{Z}} \mathbb{Q}$ induces an
isomorphism $NS_{i+1}(R;\mathbb Q)\cong NS_{i}(R_t;\mathbb Q)$.

\medskip
(c) If $k=2i=n$ and the orthogonal complement $V$ of
$I\otimes_{\mathbb Z}\mathbb Q$ in $H^{n}(R_t;\mathbb Q)$ is not
of pure Hodge type $(n/2,n/2)$, then $NS_{i}(R_t;\mathbb
Z)\subseteq I$, and the map $i^{\star}_{n}\otimes_{ \mathbb{Z}}
\mathbb{Q}$ induces an isomorphism $NS_{i+1}(R;\mathbb Q)\cong
NS_{i}(R_t;\mathbb Q)$.
\end{theorem}
\begin{proof}
(a)  From Borel-Moore homology exact sequence:
$0=H_{k+2}^{BM}({\text{Sing}(R)};\mathbb Z)\to
H_{k+2}^{BM}(R;\mathbb Z) \to H_{k+2}^{BM}(R\backslash
{\text{Sing}(R)};\mathbb Z)\to
H_{k+1}^{BM}({\text{Sing}(R)};\mathbb Z)=0$ (\cite{Fulton}, p.
371, and \cite{Fulton2}, p. 219, Lemma 3) we see that
$H_{k+2}(R;\mathbb Z) \cong H_{k+2}^{BM}(R\backslash
{\text{Sing}(R)};\mathbb Z)$ (recall that in the projective case
Borel-Moore and singular homology agree (\cite{Fulton2}, p. 217)).
On the other hand by (\cite{Fulton2}, p. 217, (26)) we have
$H_{k+2}^{BM}(R\backslash {\text{Sing}(R)};\mathbb Z) \cong
H^{2n-k}(R\backslash {\text{Sing}(R)};\mathbb Z)$, and so
$H_{k+2}(R;\mathbb Z) \cong H^{2n-k}(R\backslash
{\text{Sing}(R)};\mathbb Z)$. Therefore  $i^{\star}_{k}$
identifies with the pull-back $H^{2n-k}(R\backslash
{\text{Sing}(R)};\mathbb Z)\to H^{2n-k}(R_t;\mathbb Z)$. Now, by
the Lefschetz Theorem with Singularities (see \cite{GMP3}, p. 199
or also \cite{Hamm}, p. 552) we know that the pair $(R\backslash
{\text{Sing}(R)},R_t)$ is $n$-connected (\cite{Spanier}, p. 373).
From the relative Hurewicz Isomorphism Theorem and the Universal
Coefficient Theorem (\cite{Spanier}, p. 397 and p. 243) it follows
that $H^{2n-k}(R\backslash {\text{Sing}(R)},R_t;\mathbb Z)=0$ for
$2n-k\leq n$, and that $H^{n+1}(R\backslash
{\text{Sing}(R)},R_t;\mathbb Z)$ is torsion-free. This implies
that $H_{k+2}(R;\mathbb Z) \cong H^{2n-k}(R\backslash
{\text{Sing}(R)};\mathbb Z)\to H^{2n-k}(R_t;\mathbb Z)$ is an
isomorphism for $2n-k<n$, and injective with torsion-free cokernel
when $k=n$.

It remains to prove that $H_{n+2}(R;\mathbb Z)=I$. Since
$i^{\star}_n$ is injective with torsion free cokernel, it will
suffice to prove that the space $ I\otimes_{\mathbb Z}\mathbb
Q\subseteq H^n(R_t;\mathbb Q)$ of invariants with rational
coefficients is equal to the image of the injective map
$i^{\star}_n\otimes_{\mathbb Z}\mathbb Q:H_{n+2}(R;\mathbb Q)\to
H^n(R_t;\mathbb Q)$. This is a consequence of Deligne Invariant
Subspace Theorem (\cite{PS}, p. 165) in view of the following
reasoning.

Let $\widetilde R\to R$ be a desingularization of $R$, and let $L$
be a general pencil of hyperplane sections of $R$. Denote by
$\widetilde R_L$ the blowing-up of $\widetilde R$ along the base
locus $B_L$ of $L$, and consider the induced maps $\widetilde
R_L\to\widetilde R\to R$. By Decomposition Theorem (\cite{Dimca2},
Proposition 5.4.4 p. 157, and Corollary 5.4.11 p. 161)  we know
that $H^{n+2}(R;\mathbb Q)$   is naturally embedded in
$H^{n+2}(\widetilde R;\mathbb Q)$ via pull-back. Therefore the
push-forward $H_{n+2}(\widetilde R;\mathbb Q)\to H_{n+2}(R;\mathbb
Q)$ is surjective. We deduce that the image of
$H_{n+2}(R;\mathbb{Q})$ in $H^{n}(R_t;\mathbb{Q})$ is equal to the
image of $H_{n+2}(\widetilde R;\mathbb{Q})$ via Gysin map composed
with  Poincar\'e duality $H_{n+2}(\widetilde R;\mathbb{Q})\to
H_n(R_t;\mathbb Q)\cong H^n(R_t;\mathbb Q)$. On the other hand by
the decomposition $ H_{n+2}(\widetilde
R_L;\mathbb{Q})=H_{n+2}(\widetilde R;\mathbb{Q})\oplus
H_{n}(B_L;\mathbb{Q}) $ (\cite{Voisin}, p. 170, Th\'eor\`eme 7.31)
we see that the image of $H_{n+2}(\widetilde R_L;\mathbb{Q})$ in
$H^{n}(R_t;\mathbb{Q})$ is equal to the image of
$H_{n+2}(\widetilde R;\mathbb{Q})$ plus the image of the
push-forward $H_{n}(B_L;\mathbb{Q})\to H_{n}(R_t;\mathbb{Q})\cong
H^{n}(R_t;\mathbb{Q})$. But this last image is contained in the
image of $H_{n+2}(R;\mathbb{Q})$: in fact by Poincar\'e duality
and Lefschetz Hyperplane Theorem we have
$H_{n}(B_L;\mathbb{Q})\cong H^{n-2}(B_L;\mathbb{Q})\cong
H^{n-2}(R_t;\mathbb{Q})\cong H_{n+2}(R_t;\mathbb{Q})$, so
$H_{n}(B_L;\mathbb{Q})$ arrives in $H^{n}(R_t;\mathbb{Q})$ passing
through the push-forward $H_{n+2}(R_t;\mathbb{Q})\to
H_{n+2}(R;\mathbb{Q})$. It follows that the image of
$H_{n+2}(R;\mathbb{Q})$ in $H^{n}(R_t;\mathbb{Q})$ is equal to the
image of $H_{n+2}(\widetilde R_L;\mathbb{Q})$ in
$H^{n}(R_t;\mathbb{Q})$. Taking into account the Poincar\'e
duality $H_{n+2}(\widetilde R_L;\mathbb{Q})\cong H^{n}(\widetilde
R_L;\mathbb{Q})$, this image is the invariant space
$I\otimes_{\mathbb Z}\mathbb Q$ by the quoted Deligne Theorem.

\smallskip (b) In view of previous property, we only have to prove that
the map $NS_{i+1}(R;\mathbb Q)$ $\to NS_{i}(R_t;\mathbb Q)$ is
surjective for any $n< k=2i\leq 2n$. We argue by induction on $n$
and by decreasing induction on $k$, the cases $n=1$ and $k=2n$
being trivial. Hence assume $n>1$ and $n< k=2i< 2n$. As before,
let  $L:=\{R_t\}_{t\in {\mathbb{P}^1}}$ be a general pencil of
hyperplane sections of $R$, and denote by ${R}_L$ the blowing-up
of $R$ at the base locus $B_L$. By previous property (a) we know
that all cycles in  $H_{2i}(R_t;\mathbb{Z})$ are invariant. Based
on this, the following argument proves that
$NS_{i+1}({R}_L;\mathbb Q)$ maps onto $NS_{i}(R_t;\mathbb Q)$
(compare with \cite{Moishezon}, p. 242, Lemma 2 and proof).

In fact, fix any algebraic class $\xi\in  NS_{i}(R_t;\mathbb Q)$,
which we may assume represented by some projective algebraic
subvariety $S_1\subseteq R_t$ of dimension $i$, and consider the
Hilbert scheme $\mathcal S$, with reduced structure, parametrizing
pairs $(S, R_{t'})$, with  $ R_{t'}$ a hyperplane section  of $R$,
and $S\subseteq  R_{t'}$ a projective subvariety of dimension $i$.
Let $\mathcal C \subseteq \mathcal S$ be an irreducible projective
curve passing through the point $(S_1, R_t)$. Since $R_t$ is
Noether-Lefschetz general, we may assume $\mathcal C$ dominating
$L$ and such that $t$ is a regular value of the natural branched
covering map $\pi:\mathcal C \to L$. The fibres of $\pi$ sweep out
a projective subvariety $T\subseteq R_L$ of dimension $i + 1$,
whose intersection with $R_t$ is the union of all the subvarieties
$S_h$, $h = 1,\dots,p$, corresponding to the fibre of $\pi$ over
the point $t\in L$ ($p$ = degree of $\pi$). Since the monodromy of
$\pi$ is transitive, by (a) we deduce that all the $S_h$ are
homologous in $R_t$, and therefore $\xi$ comes from
$\frac{1}{p}\cdot T$ through the natural map
$NS_{i+1}({R}_L;\mathbb Q)\to NS_{i}(R_t;\mathbb Q)$.  This proves
that $NS_{i+1}({R}_L;\mathbb Q)$ maps onto $NS_{i}(R_t;\mathbb
Q)$.

On the other hand $NS_{i+1}(R;\mathbb Q)\oplus NS_{i+1}(B_L\times \Ps^1;\mathbb Q)$ maps onto $NS_{i+1}({R}_L;\mathbb Q)$  (\cite{Fulton}, Proposition 6.7, (e), p. 115).  Hence $NS_{i+1}(R;\mathbb
Q)\oplus NS_{i}(B_L;\mathbb Q)$ maps onto $NS_{i}({R}_t;\mathbb
Q)$. Now by induction $NS_{i+1}(R_t;\mathbb Q)$ maps onto
$NS_{i}(B_L;\mathbb Q)$, and $NS_{i+2}(R;\mathbb Q)$ maps onto
$NS_{i+1}(R_t;\mathbb Q)$. This means that the cycles of
$NS_{i}(B_L;\mathbb Q)$ arrive in $NS_{i}(R_t;\mathbb Q)$ as
cycles coming from $NS_{i+2}(R;\mathbb Q)$, hence as cycles coming
from $NS_{i+1}(R;\mathbb Q)$.

\smallskip (c) Now assume $k=2i=n$. Since $NS_{n/2}(R_t;\mathbb Q)$
is globally invariant (\cite{Lew}, p. 207, Theorem 13.18  and
proof), $V$ is not of pure Hodge type $(n/2,n/2)$, and $V$  is
irreducible (\cite{DGF}, Theorem 3.1), by a standard argument
(compare e.g. with \cite{IJM2}, proof of Theorem 1.1) it follows
that $NS_{n/2}(R_t;\mathbb Q)\subseteq I\otimes_{\mathbb Z}\mathbb
Q$. Then previous argument we used in proving (b) works well again
to prove that $NS_{n/2+1}(R;\mathbb Q)\cong NS_{n/2}(R_t;\mathbb
Q)$ (in this case $NS_{n/2+1}(R_t;\mathbb Q)$ maps onto
$NS_{n/2}(B_L;\mathbb Q)$, and $NS_{n/2+2}(R;\mathbb Q)$ maps onto
$NS_{n/2+1}(R_t;\mathbb Q)$ by (b)). Finally we notice that the
inclusion $NS_{n/2}(R_t;\mathbb Q)\subseteq I\otimes_{\mathbb
Z}\mathbb Q$ implies that $NS_{n/2}(R_t;\mathbb Z)\subseteq I$ for
$H^n(R_t;\mathbb Z)\slash I$ is torsion free by (a).
\end{proof}

\bigskip
\begin{remark} We will not need this fact but a similar
argument as before shows that {\it all cycles in
$H_{k}(R_t;\mathbb{Z})$ are invariant also for $0\leq k<n$, and if
$0\leq k=2i\leq n$ is even and $h^{n,0}(R_t)>h^{n,0}(\widetilde
R)$, then $i^{\star}_{k}\otimes_{ \mathbb{Z}} \mathbb{Q}$ induces
a surjection of $NS_{i+1}(R;\mathbb Q)$ onto $NS_{i}(R_t;\mathbb
Q)$} (compare with \cite{Moishezon}, p. 245, Theorem 5.4). Indeed
for $k<n$ $H_{k}(B_L;\mathbb{Z})$ maps onto
$H_{k}(R_t;\mathbb{Z})$ by Lefschetz Hyperplane Theorem. Moreover,
since $h^{n,0}(R_t)>h^{n,0}(\widetilde R)$ then
$h^{n-1,0}(B_L)>h^{n-1,0}(R_t)$, and therefore one may use
induction as in (b).
\end{remark}

\bigskip {\it{A topological description  of the blowing-up $P=Bl_W(Y)$}}.
\smallskip

We are going to prove there is a natural isomorphism
\begin{equation}\label{niso}
H_k(P;\mathbb Z)\cong H_{k-2}(W;\mathbb Z)\oplus H_k(Y;\mathbb Z)
\end{equation}
for any $k$ (see Corollary \ref{ndecscoppdue} below), from which
we deduce that  $H^m(Y;\mathbb Z)+H^m(X_t;\mathbb Z)_W$ is equal
to the image of the map $H_{m+2}(P;\mathbb{Z})\to
H^{m}(X_t;\mathbb{Z})$ obtained composing the Gysin map
$H_{m+2}(P;\mathbb{Z})\to H_{m}(X_t;\mathbb{Z})$, induced by the
natural inclusion $X_t\subset P$, with Poincar\'e duality
(Corollary \ref{nimp}). This is an intermediate step to prove
(\ref{ninv}). Recall that $P$ can have isolated singularities, so,
to prove (\ref{niso}), we cannot apply (\cite{Voisin}, p. 170,
Th\'eor\`eme 7.31). For similar results compare also with
(\cite{Fulton}, Proposition 6.7, (e), p. 115),
(\cite{Hartshorne2}, Proposition 4.5, p. 43), and (\cite{CG},
Corollary 3, p. 371). Next, comparing $P$ with $Q$, and using
Theorem \ref{finv}, we will prove (\ref{ninv}) (see Proposition
\ref{semi} below).

In order to prove (\ref{niso}), we need some preliminaries.

\begin{notations}\label{ng} (i)
Consider the natural commutative diagram:
$$
\begin{array}{ccccc}
 \widetilde W&\stackrel{j}\hookrightarrow  & P  \\
\stackrel {g}{}\downarrow &  &\stackrel {f}{}\downarrow  \\
W&\stackrel{i}\hookrightarrow  & Y\\
\end{array}
$$
where $\widetilde W=\Ps(\oc_{W}(-k)\oplus\oc_{W}(-d))$ denotes the
exceptional divisor on $P$. Put $U=P\backslash\widetilde W\cong
Y\backslash W$, and consider, for any integer $k$, the following
natural commutative diagram in Borel-Moore Homology Theory
(\cite{Fulton2}, p. 219, Exercise 5):
\begin{equation}\label{ndzero}
\begin{array}{ccccccc}
H_{k+1}^{BM}(U;\mathbb{Z})&\stackrel {\partial}{\to}&H_{k}(\widetilde W;\mathbb{Z})&\stackrel{j_*}{\to}
&H_{k}(P;\mathbb{Z})&{\to}&H_{k}^{BM}(U;\mathbb{Z})\\
\Vert & &\stackrel {g_*}{}\downarrow & & \stackrel {f_*}{}\downarrow& &\Vert \\
H_{k+1}^{BM}(U;\mathbb{Z})&\stackrel {\partial}{\to}&H_{k}(W;\mathbb{Z})&\stackrel{i_*}{\to}&H_{k}(Y;\mathbb{Z})
&{\to}&H_{k}^{BM}(U;\mathbb{Z})\\
\end{array}
\end{equation}
(as before, in the projective case, we identify Borel-Moore and
singular homology groups).

\smallskip
(ii) Besides the push-forward maps $f_*$ and $g_*$ we may consider
the Gysin maps $f^{\star}:H_{k}(Y;\mathbb{Z})\to
H_{k}(P;\mathbb{Z})$ and $ g^{\star}:H_{k-2}(W;\mathbb{Z})\to
H_{k}(\widetilde W;\mathbb{Z}) $ (see \cite{Baum}, or
\cite{Fulton}, p. 382, Example 19.2.1), and the map
$g^{!}:H_{k}(W;\mathbb{Z})\to H_{k}(\widetilde W;\mathbb{Z})$
defined composing the Gysin map $g^{\star}:H_{k}(W;\mathbb{Z})\to
H_{k+2}(\widetilde W;\mathbb{Z})$ with the cap-product
$\frown\,[P]_{|{\widetilde W}}:H_{k+2}(\widetilde W;\mathbb{Z})
{\to} H_{k}(\widetilde W;\mathbb{Z})$. Here $[P]_{|{\widetilde
W}}$ denotes the pull-back in $H^{2}(\widetilde W;\mathbb{Z})$ of
the orientation class $[P]\in H^{2}(E,E\backslash P
;\mathbb{Z})\cong H_{2m+2}(P;\mathbb{Z})$ defined by $P\subset E$
(\cite{Fulton}, p. 371 and p. 378).

\smallskip
(iii) More generally, consider a line bundle $\oc_{E}(D)$, and
denote by $[D]\in H^{2}(E;\mathbb{Z})$ its cohomology class. We
may define a map $ g^{!}_{[D]}:H_{k}(W;\mathbb{Z})\to
H_{k}(\widetilde W;\mathbb{Z}) $ in a similar way as  $g^{!}$,
using $[D]_{|{\widetilde W}}$ instead of $[P]_{|{\widetilde W}}$.
Observe that $g^{!}=g^{!}_{[P]}$ because
$\oc_{E}(P)=\oc_{E}(\Theta +k\Lambda)$ ($\Lambda :=$ pull-back of
the hyperplane section of $Y\subseteq\Ps^N$ through $E\to Y$).

\smallskip
(iv) We denote by $ \gamma:H_{k}(P;\mathbb{Z})\to
H_{k-2}(W;\mathbb{Z}) $ the map obtained composing the Gysin map
$H_{k}(P;\mathbb{Z})\to H_{k-2}(\widetilde W;\mathbb{Z})$ with
push-forward $H_{k-2}(\widetilde W;\mathbb{Z})\to
H_{k-2}(W;\mathbb{Z})$, and by $ \mu_k: H_{k}(G;\mathbb{Z})\to
H_{k-2}(W;\mathbb{Z})$ and $\nu_k: H_{k}(G;\mathbb{Z})\to
H_{k}(Y;\mathbb{Z})$ the Gysin map and the push-forward.
\end{notations}

\begin{lemma}
\label{fult} The following properties hold true.

\medskip
(a) $j_*\circ g^{!}=f^{\star}\circ i_*$, and $f_*\circ
f^{\star}={\text{id}}_{H_{k}(Y;\mathbb{Z})}$;

\medskip
(b) for any $D=\pm\Theta +l\Lambda$, $l\in \mathbb{Z}$,  one has
$g_*\circ g^{!}_{[D]}={\pm\text{id}}_{H_{k}(W;\mathbb{Z})}$; in
particular $g_*\circ g^{!}={\text{id}}_{H_{k}(W;\mathbb{Z})}$, and
$g^{!}_{[D]}$ and $g^{\star}$ are injective;

\medskip (c) $\Im(g^{\star}:H_{k-2}(W;\mathbb{Z})\to
H_{k}(\widetilde W;\mathbb{Z}))= \ker (g_*:H_{k}(\widetilde
W;\mathbb{Z})\to H_{k}(W;\mathbb{Z}))$;

\medskip
(d)  $\gamma\circ j_*\circ
g^{\star}=-{\text{id}}_{H_{k-2}(W;\mathbb{Z})}$; in particular
$j_*\circ g^{\star}$ embeds $H_{k-2}(W;\mathbb{Z})$ into
$H_{k}(P;\mathbb{Z})$;

\medskip
(e) the diagram obtained from (\ref{ndzero}) replacing $g_*$ and
$f_*$ with $g^{!}$ and $f^{\star}$ is commutative:
\begin{equation}\label{nd}
\begin{array}{ccccccc}
H_{k+1}^{BM}(U;\mathbb{Z})&\stackrel
{\partial}{\to}&H_{k}(\widetilde
W;\mathbb{Z})&\stackrel{j_*}{\to}&
H_{k}(P;\mathbb{Z})&{\to}&H_{k}^{BM}(U;\mathbb{Z})\\
\Vert & &\stackrel {g^{!}}{}\uparrow & & \stackrel {f^{\star}}{}\uparrow& &\Vert \\
H_{k+1}^{BM}(U;\mathbb{Z})&\stackrel
{\partial}{\to}&H_{k}(W;\mathbb{Z})&
\stackrel{i_*}{\to}&H_{k}(Y;\mathbb{Z})&{\to}&H_{k}^{BM}(U;\mathbb{Z}).\\
\end{array}
\end{equation}
\end{lemma}
\begin{proof}
(a) and (b) By functoriality, one may construct $f^{\star}$ in a
similar way as $g^!$, i.e. composing the Gysin map
$H_{k}(Y;\mathbb{Z})$ $\to H_{k+2}(E;\mathbb{Z})$ with the
cap-product $\frown\,[P]:H_{k+2}(E;\mathbb{Z}) {\to}
H_{k}(P;\mathbb{Z})$ (\cite{Fulton}, p. 371, (2)). Therefore the
equality $j_*\circ g^{!}=f^{\star}\circ i_*$ follows from  the
commutativity of the two little squares in the following diagram:
$$
\begin{array}{ccc}
H_k(\widetilde W;\mathbb Z)&\stackrel{j_*}{\to}& H_k(P;\mathbb Z) \\
\stackrel {\frown\,[P]}{}\uparrow \quad\quad&     &  \quad\quad\uparrow\stackrel {\frown\,[P]}{}\\
H_{k+2}(\widetilde W;\mathbb Z)&\stackrel{}{\to}& H_{k+2}(E;\mathbb Z) \\
\stackrel {}{}\uparrow &     &  \uparrow\stackrel {}{}\\
H_k(W;\mathbb Z)&\stackrel{i_*}{\to}& H_k(Y;\mathbb Z) \\
\end{array}
$$
(the lower vertical maps are the Gysin maps, and
$H_{k+2}(\widetilde W;\mathbb Z)\to H_{k+2}(E;\mathbb Z)$ is the
push-forward).

As for the map $f_*\circ f^{\star}$, first observe that it is
equal to the composition of the map $H_{k}(Y;\mathbb{Z})$ $\to
H_{k+2}(E;\mathbb{Z})$ with  $\frown\,[P]:H_{k+2}(E;\mathbb{Z})
{\to} H_{k}(E;\mathbb{Z})$ and push-forward $H_{k}(E;\mathbb{Z})$
$\to H_{k}(Y;\mathbb{Z})$. Now pick any $y\in H_{k}(Y;\mathbb{Z})$
and denote by $\tilde y$ its image in $H_{k+2}(E;\mathbb{Z})$.
Since capping $\tilde y$ with $[\Lambda]$ and pushing-forward it gives
$0\in H_{k}(Y;\mathbb{Z})$, then previous composition is the same
as composing $H_{k}(Y;\mathbb{Z})$ $\to H_{k+2}(E;\mathbb{Z})$
with $\frown\,[\Theta]:H_{k+2}(E;\mathbb{Z}) {\to}
H_{k}(E;\mathbb{Z})$ and push-forward $H_{k}(E;\mathbb{Z})$ $\to
H_{k}(Y;\mathbb{Z})$ (recall that $\oc_{E}(P)=\oc_{E}(\Theta
+k\Lambda)$). And this map is equal to
${\text{id}}_{H_{k}(Y;\mathbb{Z})}$ because $\Theta\cong Y$ is a
section of the projective bundle $E\to Y$.

For the same reason one has $g_*\circ g^{!}_{[D]}=g_*\circ
g^{!}_{[\pm \Theta]}={\pm\,\text{id}}_{H_{k}(W;\mathbb{Z})}$.

\smallskip
(c) Fix any $D=\pm \Theta +l\Lambda$. Since $[D]$ restricts to
$\pm 1\in H^{2}(g^{-1}(w);\mathbb{Z})$ for any $w\in W$, then it
determines a cohomology extension of the fiber-bundle
$g:\widetilde W\to W$ (\cite{Spanier}, p. 256). This in turn
induces, by the Leray-Hirsh Theorem (\cite{Spanier}, p. 258), a
decomposition
\begin{equation}\label{spanier}
H_{k}(\widetilde W;\mathbb{Z})\cong H_{k}(W;\mathbb{Z})\oplus
H_{k-2}(W;\mathbb{Z})
\end{equation}
given by the isomorphism $\tilde w\in H_{k}(\widetilde
W;\mathbb{Z})\to (g_*(\tilde w), g_*(\tilde w\frown
[D]_{|{\widetilde W}}))\in H_{k}(W;\mathbb{Z})\oplus
H_{k-2}(W;\mathbb{Z})$. Using (b) one sees that the inverse map
$H_{k}(W;\mathbb{Z})\oplus H_{k-2}(W;\mathbb{Z})\to
H_{k}(\widetilde W;\mathbb{Z})$ is the sum  of
$\pm\,g^{!}_{[D]}:H_{k}(W;\mathbb{Z})\to H_{k}(\widetilde
W;\mathbb{Z})$ with $\pm\,g^{\star}:H_{k-2}(W;\mathbb{Z})\to
H_{k}(\widetilde W;\mathbb{Z})$. It follows that $\Im g^{\star}=
\ker g_*$.

\smallskip
(d) Since $\oc_{P}(\widetilde W)\otimes \oc_{\widetilde W}=
\oc_{E}(-\Theta +d\Lambda)\otimes \oc_{\widetilde W}$ then
composing $j_*\circ g^{\star}: H_{k-2}(W;\mathbb{Z}) \to
H_{k}(P;\mathbb{Z})$ with Gysin map $H_{k}(P;\mathbb{Z})\to
H_{k-2}(\widetilde W;\mathbb{Z})$ we get the map $g^{!}_{[D]}$
induced by $D=-\Theta +d\Lambda$. Therefore  $\gamma\circ j_*\circ
g^{\star}=g_*\circ
g^{!}_{[D]}=-{\text{id}}_{H_{k-2}(W;\mathbb{Z})}$ by (b).

\smallskip
(e) The central square diagram in (\ref{nd}) commutes by (a).
Moreover the right square commutes by functoriality of the Gysin
maps. So we only have to prove the left square diagram commutes.
To this purpose, denote by $h_*:H_{k+1}^{BM}(U;\mathbb{Z})\to
H_{k+1}^{BM}(U;\mathbb{Z})$ the push-forward isomorphism on the
left of the diagram (\ref{ndzero}), so that we have $\partial\circ
h_*=g_*\circ\partial$. We have to prove that $\partial=g^{!}\circ
\partial\circ h_*$. Pick any $u\in H_{k+1}^{BM}(U;\mathbb{Z})$ and
let $\tilde w:=\partial(u)-(g^{!}\circ \partial\circ h_*)(u)\in
H_{k}(\widetilde W;\mathbb{Z})$. By (b) and the commutativity of
(\ref{ndzero}) we have $g_*(\tilde
w)=g_*(\partial(u))-g_*(g^{!}(\partial(h_*(u))))=g_*(\partial(u))-\partial(h_*(u))=0$.
Therefore by (c) we deduce that $\tilde w=g^{\star}(w)$ for some
$w\in H_{k-2}(W;\mathbb{Z})$. On the other hand by (a) and
(\ref{ndzero}) we see that $j_*(g^{\star}(w))=j_*(\tilde w
)=j_*(\partial(u))-(j_*\circ g^{!}\circ
\partial\circ h_*)(u)=-(f^{\star}\circ i_*\circ
\partial\circ h_*)(u)=0$. This implies $\tilde w=0$ because by (d) we know that $j_*\circ g^{\star}$ is
injective.
\end{proof}

\begin{proposition}
\label{ndecscopp} The  sequence
$$
0\to H_{k}(W;\mathbb{Z})\stackrel
{\alpha}{\longrightarrow}H_{k}(\widetilde W;\mathbb{Z})\oplus
H_{k}(Y;\mathbb{Z})\stackrel
{\beta}{\longrightarrow}H_{k}(P;\mathbb{Z})\to 0
$$
with $\alpha(w)=(g^{!}(w),-i_{*}(w))$ and $\beta(\tilde
w,y)=j_*(\tilde w)+f^{\star}(y)$, is exact.
\end{proposition}
\begin{proof} The map $\alpha$ is injective because $g^{!}$ is
by Lemma \ref{fult}, (b). To prove that $\beta$ is surjective, fix
$p\in H_{k}(P;\mathbb{Z})$. From the commutativity of
(\ref{ndzero}) and (\ref{nd}) (Lemma \ref{fult}, (e)), it follows
that the image of $(f^{\star}\circ f_*)(p)-p$ in
$H_{k}^{BM}(U;\mathbb{Z})$ is $0$. Hence there exists $\tilde w\in
H_{k}(\widetilde W;\mathbb{Z})$ such that $j_*(\tilde
w)=(f^{\star}\circ f_*)(p)-p$, i.e. $p=(f^{\star}\circ
f_*)(p)-j_*(\tilde w)=\beta(-\tilde w,f_*(p))$. This proves that
$\beta$ is surjective. Finally observe that $\beta\circ \alpha=0$
by the commutativity of (\ref{nd}), and that $\ker(\beta)\subseteq
\Im(\alpha)$ by diagram chasing in (\ref{nd}).
\end{proof}

\begin{corollary}
\label{ndecscoppdue} The map $(w,y)\in H_{k-2}(W;\mathbb{Z})\oplus
H_{k}(Y;\mathbb{Z})\to j_*(g^{\star}(w))+f^{\star}(y)\in
H_{k}(P;\mathbb{Z}) $ is an isomorphism, whose inverse map  is
given by $p\in H_{k}(P;\mathbb{Z}) \to (-\gamma(p),f_*(p))\in
H_{k-2}(W;\mathbb{Z})\oplus H_{k}(Y;\mathbb{Z})$. In particular,
via this isomorphism, the push-forward $H_{k}(G;\mathbb{Z})\to
H_{k}(P;\mathbb{Z})$ identifies with the map $x\in
H_{k}(G;\mathbb{Z})\to (-\mu_{k}(x),\nu_{k}(x))\in
H_{k-2}(W;\mathbb{Z})\oplus H_{k}(Y;\mathbb{Z})$.
\end{corollary}
\begin{proof} By Proposition \ref{ndecscopp} we know that for any $p\in H_{k}(P;\mathbb{Z})$
there are $\tilde w\in H_{k}(\widetilde W;\mathbb{Z})$ and $y\in
H_{k}(Y;\mathbb{Z})$ such that $p= j_*(\tilde w)+f^{\star}(y)$. By
(\ref{spanier}) we may write $\tilde w=g^{!}(w_1)+g^{\star}(w_2)$
for suitable $w_1\in H_{k}(W;\mathbb{Z})$ and $w_2\in
H_{k-2}(W;\mathbb{Z})$. Therefore $p=
j_*(g^{\star}(w_2))+j_*(g^{!}(w_1))+f^{\star}(y)$, and by Lemma
\ref{fult}, (a), we get $p=
j_*(g^{\star}(w_2))+f^{\star}(y+i_*w_1)$. This proves that the
given map is onto. Moreover if $j_*(g^{\star}(w))+f^{\star}(y)=0$
then by Proposition \ref{ndecscopp} we may write
$(g^{\star}(w),y)=(g^{!}(u),-i_*(u))$ for some $u\in
H_{k}(W;\mathbb{Z})$. Again by (\ref{spanier}) we deduce
$g^{\star}(w)=g^{!}(u)=0$, and so $u=0$ by the injectivity of
$g^{!}$ (Lemma \ref{fult}, (b)). This proves that the given map is
injective. As for the description of its inverse, it follows from
Lemma \ref{fult}, (a) and (d).
\end{proof}

\begin{corollary}
\label{nimp} Let $X_t\in |H^0(Y,\ic_{W,Y}(d))|$ be a smooth
hypersurface containing $W$. Let $H_{m+2}(P;\mathbb{Z})\to
H^{m}(X_t;\mathbb{Z})$ be the map obtained composing the Gysin map
$H_{m+2}(P;\mathbb{Z})\to H_{m}(X_t;\mathbb{Z})$ (induced by the
natural inclusion $X_t\hookrightarrow P$) with Poincar\'e duality.
Then the image of this map is equal to
$H^m(Y;\mathbb{Z})+H^{m}(X_t;\mathbb{Z})_W$.
\end{corollary}

\medskip
Now consider the following commutative diagrams:

\begin{equation}\label{pqx}
\begin{array}{ccccc}
H_{m+2}(P;\mathbb{Z})&\\
 \downarrow &\searrow & &\\
H_{m+2}(Q;\mathbb{Z})& \stackrel{}{\rightarrow}& H^{m}(X_t;\mathbb{Q})\\
\end{array}
\end{equation}

\medskip\noindent
and

\begin{equation}\label{nndzero}
\begin{array}{ccccccc}
H_{k}(G;\mathbb{Z})&\stackrel
{}{\to}&H_{k}(P;\mathbb{Z})&\stackrel{}{\to}
&H_{k}(P,G;\mathbb{Z})&{\to}&H_{k-1}(G;\mathbb{Z})\\
\downarrow  & &\stackrel {}{}\downarrow & & \stackrel {}{}\Vert& &\downarrow  \\
H_{k}(\{v_{\infty}\};\mathbb{Z})&\stackrel
{}{\to}&H_{k}(Q;\mathbb{Z})&\stackrel{}{\to}&H_{k}(Q,\{v_{\infty}\};\mathbb{Z})
&{\to}&H_{k-1}(\{v_{\infty}\};\mathbb{Z}).\\
\end{array}
\end{equation}

\medskip
\noindent In diagram (\ref{pqx}) the vertical map is the
push-forward corresponding to the natural projection $P\to Q$, and
the other maps are obtained composing the Gysin maps
$H_{m+2}(P;\mathbb{Z})\to H_{m}(X_t;\mathbb{Z})$ and
$H_{m+2}(Q;\mathbb{Z})\to H_{m}(X_t;\mathbb{Z})$  with Poincar\'e
duality $H_{m}(X_t;\mathbb{Z})\cong H^{m}(X_t;\mathbb{Z})$. The
rows appearing in diagram (\ref{nndzero}) are the exact sequences
of the pairs $(P,G)$  and $(Q,\{v_{\infty}\})$, and the vertical
maps denote push-forward (compare with \cite{La}, p. 23).

Combining Theorem \ref{finv}, (a), and Corollary \ref{nimp} with
(\ref{pqx}), we see that (\ref{ninv}) holds if and only if
$H_{m+2}(P;\mathbb{Z})$ maps onto $H_{m+2}(Q;\mathbb{Z})$. On the
other hand, by diagram (\ref{nndzero}) with $k=m+2$, we deduce
that $H_{m+2}(P;\mathbb{Z})\to H_{m+2}(Q;\mathbb{Z})$ is
surjective if and only if the push-forward
$H_{m+1}(G;\mathbb{Z})\to H_{m+1}(P;\mathbb{Z})$ is injective, and
by Corollary \ref{ndecscoppdue} this is equivalent to say that
$\ker \mu_{m+1}\cap\ker\nu_{m+1}=0$. We notice that, in the case
of rational coefficients, Hard Lefschetz Theorem implies that both
maps $\mu_{m+1}\otimes_{\mathbb Z}\mathbb Q$ and
$\nu_{m+1}\otimes_{\mathbb Z}\mathbb Q$ are injective, i.e. $\ker
\mu_{m+1}$ and $\ker\nu_{m+1}$ are finite torsion groups. Summing
up we have the following:

\begin{proposition}
\label{semi} $I_W(\mathbb Z)=H^m(Y;\mathbb Z)+H^m(X_t;\mathbb
Z)_W$ if and only if $\ker \mu_{m+1}\cap\ker\nu_{m+1}=0$. In
particular $I_W(\mathbb Q)=H^m(Y;\mathbb Q)+H^m(X_t;\mathbb Q)_W$.
\end{proposition}

Unfortunately we  are not able to prove that $\ker
\mu_{m+1}\cap\ker\nu_{m+1}=0$ in this generality. This has
prevented us from extending (\cite{IJM2}, Theorem 1.2) in the case
$Y$ is not necessarily a complete intersection. In general it may
happen $\ker\nu_{m+1}\neq 0$, so we expect that $\ker\mu_{m+1}=0$
(which holds true when $W$ is smooth). Since $\ker\nu_{m+1}$ is a
torsion group, it would be sufficient to prove that $\mu_{m+1}$
simply injects the torsion. We will overcome this difficulty later
on, in the proof of Theorem \ref{maintheorem}, assuming
$X_t$ varies in the linear system $|H^0(Y,\ic_{Z,Y}(d))|$ (see
Lemma \ref{gmp} below).

\medskip We conclude this section identifying
the kernel of the push-forward $H_{m}(W;\mathbb Z)\to
H_{m}(X_t;\mathbb Z)$, and the intersection $H^{m}(Y;\mathbb
Z)\cap H_{m}(X_t;\mathbb Z)_W$ in $H^{m}(X_t;\mathbb Z)$. We need
again some preliminaries, the first one is the following:

\begin{lemma}
\label{semid} The Gysin map $\mu_{m+2}: H_{m+2}(G;\mathbb{Z})\to
H_{m}(W;\mathbb{Z})$ is injective.
\end{lemma}
\begin{proof} Consider the
following natural commutative diagram
$$
\begin{array}{ccccc}
 H^{m-2}(W;\mathbb Z)&&  \\
\downarrow &\searrow & & \\
 H_{m}(W;\mathbb Z) & \longrightarrow & H^{m-2}(W';\mathbb Z)\cong H_{m-2}(W';\mathbb Z)& &\\
\end{array}
$$
where the vertical map $H^{m-2}(W;\mathbb Z)\to H_{m}(W;\mathbb
Z)$ is the duality morphism (\cite{crory}, p. 150), $W'$ is the
general hyperplane section of $W$ ($W'$ is smooth),
$H^{m-2}(W;\mathbb Z)\to H^{m-2}(W';\mathbb Z)$ is the pull-back,
and $H_{m}(W;\mathbb Z)\to H^{m-2}(W';\mathbb Z)$ is the Gysin map
$H_{m}(W;\mathbb Z)\to H_{m-2}(W';\mathbb Z)$ composed with with
Poincar\'e duality $H_{m-2}(W';\mathbb Z)\cong H^{m-2}(W';\mathbb
Z)$. By Lefschetz Hyperplane Theorem with Singularities
(\cite{GMP3}, \cite{Hamm}) we know that $H^{m-2}(W;\mathbb Z)\to
H^{m-2}(W';\mathbb Z)$ is injective. Therefore also the duality
morphism is. This proves our claim, because $\mu_{m+2}$ identifies
with the duality morphism in view of Poincar\'e duality
$H_{m+2}(G;\mathbb Z)\cong H^{m-2}(G;\mathbb Z)$ and Lefschetz
Hyperplane Theorem again $H^{m-2}(G;\mathbb Z)\cong
H^{m-2}(W;\mathbb Z)$.
\end{proof}

\medskip
Now consider the following commutative diagram:
$$
\begin{array}{ccccc}
H_{m+2}(G;\mathbb Z) &\stackrel{\nu_{m+2}}\to & H_{m+2}(Y;\mathbb Z) \cong H^{m}(Y;\mathbb Z)\\
\stackrel {\mu_{m+2}}{} \downarrow & & \stackrel {}{}\downarrow &\\
H_{m}(W;\mathbb Z) & \stackrel{}{\to}&
H_{m}(X_t;\mathbb Z) \cong H^{m}(X_t;\mathbb Z),\\
\end{array}
$$
where the bottom map denotes push-forward, and the right vertical
map pull-back identified with Gysin map via Poincar\'e duality. By
Lemma \ref{semid} we know that $\mu_{m+2}$ is injective. Hence we
obtain a natural inclusion
\begin{equation}\label{inuno}
\ker \nu_{m+2}\hookrightarrow \ker\left(H_{m}(W;\mathbb Z)\to
H_{m}(X_t;\mathbb Z)\right).
\end{equation}
Observe that, by Hard Lefschetz Theorem, $\ker \nu_{m+2}$ is a
torsion group.

Next consider the map $H^{m-2}(Y;\mathbb Z)\to H^{m}(X_t;\mathbb
Z)$ given composing:
$$
H^{m-2}(Y;\mathbb Z)\cong H_{m+4}(Y;\mathbb Z)\stackrel{\frown [G]
}{\longrightarrow} H_{m+2}(Y;\mathbb Z)\cong H^m(Y;\mathbb
Z)\subseteq H^m(X_t;\mathbb Z).
$$
We may obtain this map also composing
$$
H^{m-2}(Y;\mathbb Z)\cong H_{m+4}(Y;\mathbb Z)\stackrel{\frown [W]
}{\longrightarrow} H_{m}(W;\mathbb Z)\to H_m(X;\mathbb Z)\cong
H^m(X_t;\mathbb Z).
$$
Hence  $\Im\left(H^{m-2}(Y;\mathbb Z)\to H^{m}(X_t;\mathbb
Z)\right)$ is a subgroup of $H^{m}(Y;\mathbb Z) \cap
H_{m}(X;\mathbb Z)_W$. On the other hand, by  Lefschetz Hyperplane
Theorem we have $H^{m-2}(G;\mathbb Z)\cong H^{m-2}(Y;\mathbb Z)$.
Therefore, via Poincar\'e duality, there is a natural isomorphism
$\Im\, \nu_{m+2}\cong \Im\left(H^{m-2}(Y;\mathbb Z)\to
H^{m}(X_t;\mathbb Z)\right)$. Summing up we get another natural
inclusion
\begin{equation}\label{indue}
\Im\,\nu_{m+2}\hookrightarrow H^{m}(Y;\mathbb Z) \cap
H_{m}(X;\mathbb Z)_W.
\end{equation}
Observe that,  tensoring  with $\mathbb Q$, the map
$H^{m-2}(Y;\mathbb Q)\to H^{m}(X_t;\mathbb Q)$ becomes injective
by Hard Lefschetz Theorem. So we have $H^{m-2}(Y;\mathbb Q)\cong
\Im\,(\nu_{m+2}\otimes_{\mathbb Z}\mathbb Q)$ and therefore
\begin{equation}\label{intre}
H^{m-2}(Y;\mathbb Q)\hookrightarrow H^{m}(Y;\mathbb Q) \cap
H_{m}(X;\mathbb Q)_W.
\end{equation}
Actually all previous inclusions (\ref{inuno}), (\ref{indue}) and
(\ref{intre}) are equalities. This is the content of the following
Proposition \ref{semit}.

\medskip
\begin{proposition}
\label{semit} Let $\nu_{m+2}: H_{m+2}(G;\mathbb{Z})\to
H_{m+2}(Y;\mathbb{Z})$ be the push-forward. Then there are
canonical isomorphisms $\ker \nu_{m+2}\cong \ker\left(
H_{m}(W;\mathbb Z)\to H_{m}(X_t;\mathbb Z)\right)$ and
$\Im\,\nu_{m+2}\cong H^{m}(Y;\mathbb Z) \cap H_{m}(X;\mathbb
Z)_W$. In particular the push-forward $H_{m}(W;\mathbb Q)$ $\to
H_{m}(X_t;\mathbb Q)$ is injective (so $H_{m}(W;\mathbb Q)\cong
H^{m}(X;\mathbb Q)_W$), and $H^{m-2}(Y;\mathbb Q)\cong
H^{m}(Y;\mathbb Q) \cap H_{m}(X;\mathbb Q)_W$.
\end{proposition}
\begin{proof} By Theorem \ref{finv} and (\ref{pqx}) we see that the maps
$H_{m+2}(P;\mathbb Z)\to H^{m}(X_t;\mathbb Z)$ and
$H_{m+2}(P;\mathbb Z)\to H_{m+2}(Q;\mathbb Z)$ have the same
kernel. Therefore, by (\ref{nndzero}), Corollary
\ref{ndecscoppdue}, Corollary \ref{nimp} and Lemma \ref{semid}, we
deduce a natural isomorphism
$$
H^{m}(Y;\mathbb Z)+ H_{m}(X;\mathbb Z)_W\cong\left[
H^{m}(Y;\mathbb Z)\oplus H_{m}(W;\mathbb Z)\right]\slash
H_{m+2}(G;\mathbb Z)
$$
where the inclusion $H_{m+2}(G;\mathbb Z)\subseteq H^{m}(Y;\mathbb
Z)\oplus H_{m}(W;\mathbb Z)$ is defined via the map $x\to
(\nu_{m+2}'(x),-\mu_{m+2}(x))$ (here $\nu_{m+2}'$ denotes the map
$\nu_{m+2}$ composed with Poincar\'e duality
$H_{m+2}(Y;\mathbb{Z})\cong H^{m}(Y;\mathbb{Z})$). So the kernel
of the push-forward $H_{m}(W;\mathbb Z)\to H_{m}(X_t;\mathbb Z)$
identifies with the kernel of the map
$$
w\in H_{m}(W;\mathbb Z)\to [(0,w)]\in \left[ H^{m}(Y;\mathbb
Z)\oplus  H_{m}(W;\mathbb Z)\right] \slash  H_{m+2}(G;\mathbb Z).
$$
Now if $[(0,w)]= [(0,0)]$ then there is $x\in H_{m+2}(G;\mathbb
Z)$ such that $\nu_{m+2}'(x)=0$ and $w=-\mu_{m+2}(x)$. This means
that the inclusion (\ref{inuno}) is also surjective.

As for the intersection, first observe that the inclusion
(\ref{indue}) identifies with the image in $H^{m}(Y;\mathbb Z)\cap
H_{m}(X;\mathbb Z)_W$ of the following map:
\begin{equation}\label{map}
x\in H_{m+2}(G;\mathbb Z)\to[ (\nu_{m+2}'(x),0)]\in
H^{m}(Y;\mathbb Z)\cap H_{m}(X;\mathbb Z)_W.
\end{equation}
If $[ (y,0)]\in H^{m}(Y;\mathbb Z)\cap H_{m}(X;\mathbb Z)_W$ then
for some $w\in H_{m}(X;\mathbb Z)_W$ and $x\in H_{m+2}(G;\mathbb
Z)$ we have $(y,0)=(0,w)+(\nu_{m+2}'(x),-\mu_{m+2}(x))$, hence
$y=\nu_{m+2}'(x)$. This proves that the map (\ref{map}) is
surjective, so also (\ref{indue}) is.
\end{proof}

\bigskip
\section{The proof of Theorem \ref{maintheorem}.}

\begin{notations}\label{duenotazioni}
(i) Let $Y\subseteq \Ps^N$  be a smooth complex projective variety
of dimension $m+1=2r+1\geq 3$, $Z$ be a closed subscheme of $Y$,
and $\delta$ be a positive integer such that $\mathcal
I_{Z,Y}(\delta)$ is generated by global sections. Assume that for
any $k\geq \delta$ the general divisor in $|H^0(Y,\ic_{Z,Y}(k))|$
is smooth. This is equivalent to say  that $Z$ verifies condition
(0.1) in \cite{OS}. In particular $2\dim Z\leq m$.

\medskip
(ii) Let $X,G_1,\dots,G_r$ be general divisors  with $X \in
|H^0(Y,\ic_{Z,Y}(d))|$ and $G_l\in |H^0(Y,\ic_{Z,Y}(k_l))|$ for
$1\leq l\leq r$. Assume that $d>k_l\geq \delta$ for any $1\leq
l\leq r$. By (\cite{OS}, 1.2. Theorem) we know that $X,G_1,\dots,
G_r$ is a regular sequence, verifying the following conditions for
any $2\leq l\leq r$:
$$
\dim\,{\text{Sing}}(G_1\cap \dots\cap G_l)\leq l-2
\quad{\text{and}}\quad \dim\,{\text{Sing}}(X\cap G_1\cap \dots\cap
G_l)\leq l-1.
$$
Put $\Delta:=X\cap G_1\cap\dots\cap G_r$. Hence $\Delta$ is a
complete intersection of dimension $r$ containing $Z$. Denote by
$C_1,\dots,C_{\omega}$ the irreducible components of $\Delta$.
Observe that also $\Delta$ verifies condition (0.1) in \cite{OS}.
Put $W:=X\cap G_1$.

\medskip
(iii) Let $I_{Z}(\mathbb Z)$  be the subgroup  of the invariant
cocycles in $H^m(X_t;\mathbb Z)$  with respect to the monodromy
representation on $H^m(X_t;\mathbb Z)$  for the family of smooth
divisors $X_t\in |H^0(Y,\mathcal O_Y(d))|$ containing ${Z}$.
Denote by $H^m(X_t;\mathbb Z)_{Z}$ the image of $H_m(Z;\mathbb Z)$
in $H^m(X_t;\mathbb Z)$ via the natural map $H_m(Z;\mathbb Z)\to
H_m(X_t;\mathbb Z)\cong H^m(X_t;\mathbb Z)$. One may give a
similar definition for $I_{\Delta}(\mathbb Z)$, $H^m(X_t;\mathbb
Z)_{\Delta}$, $I_{W}(\mathbb Z)$, and $H^m(X_t;\mathbb Z)_{W}$,
and also with $\mathbb Q$ instead of $\mathbb Z$. Observe that
since $Z\subseteq \Delta\subseteq W$ then the monodromy group of
the family of smooth divisors in $|H^0(Y,\mathcal O_Y(d))|$
containing $W$ is a subgroup of the monodromy group of the family
of smooth divisors in $|H^0(Y,\mathcal O_Y(d))|$ containing
$\Delta$, which in turn is a subgroup of the monodromy group of
the family of smooth divisors in $|H^0(Y,\mathcal O_Y(d))|$
containing $Z$. Therefore we have
\begin{equation}\label{scatola}
I_{Z}(\mathbb Z)\subseteq I_{\Delta}(\mathbb Z)\subseteq
I_{W}(\mathbb Z).
\end{equation}
Denote by $V_{\Delta}:=I_{\Delta}(\mathbb Q)^{\perp}$  the
orthogonal complement of $I_{\Delta}(\mathbb Q)$  in
$H^m(X_t;\mathbb Q)$.

\medskip
(iv) For any $1\leq l\leq r-1$ fix general divisor $H_l\in
|H^0(Y,\mathcal O_Y(\mu_{l}))|$, with $0\ll \mu_1\ll\dots\ll
\mu_{r-1}$, and for any $0\leq l\leq r-1$ define
$(Y_l,Z_l,X_l,W_l,\Delta_l)$ as follows. For $l=0$ put
$(Y_0,Z_0,X_0,W_0,\Delta_0):=(Y,Z,X,W,\Delta)$. For $1\leq l\leq
r-1$ put $Y_l:=G_1\cap\dots\cap G_l\cap H_1\cap\dots\cap H_l$,
$Z_l:=Z\cap H_1\cap\dots\cap H_l$, $X_l:=X\cap Y_l$, $W_l:=X\cap
Y_l\cap G_{l+1}$, and $\Delta_l:=\Delta\cap Y_l$. Notice that
$\dim\,Y_{r-1}=3$ and that $\Delta_{r-1}=W_{r-1}$.
\end{notations}

\begin{lemma}
\label{gmp} Let $X \in |H^0(Y,\ic_{Z,Y}(d))|$ be a general divisor
containing $Z$. Then  the Gysin map $\mu_{m+1}:
H_{m+1}(G_1;\mathbb Z)\to H_{m-1}(W;\mathbb Z)$ is injective, and
therefore one has $I_W(\mathbb Z)=H^m(Y;\mathbb Z)+H^m(X_t;\mathbb
Z)_W$.
\end{lemma}
\begin{proof} Since $|H^0(Y,\ic_{Z,Y}(d))|$ is very ample on
$G_1\backslash Z$, then by Lefschetz Theorem with Singularities
(\cite{GMP3}, p. 199) we know that the pull-back
$H^{m-1}(G_1\backslash {Z};\mathbb Z)\to H^{m-1}(W\backslash
{Z};\mathbb Z)$ is injective for a general
$X\in|H^0(Y,\ic_{Z,Y}(d))|$ (recall that $W=G_1\cap X$). Moreover
by Bertini Theorem we have ${\text{Sing}(W)}\subseteq Z$, and
therefore $W\backslash Z\subseteq W\backslash {\text{Sing}(W)}$.
So we may consider the following natural commutative diagram
$$
\begin{array}{ccccc}
H^{m-1}(G_1;\mathbb Z)&\stackrel{}\rightarrow  & H^{m-1}(W\backslash {\text{Sing}(W)};\mathbb Z)  \\
\stackrel {}{}\downarrow &  &\stackrel {}{}\downarrow  \\
H^{m-1}(G_1\backslash Z;\mathbb Z) &\stackrel{}\hookrightarrow  & H^{m-1}(W\backslash Z;\mathbb Z)\\
\end{array}
$$
where all maps are pull-back. Since also $H^{m-1}(G_1;\mathbb
Z)\to H^{m-1}(G_1\backslash Z;\mathbb Z)$ is injective
(\cite{Dimca2}, Theorem 5.4.12, p. 162), we deduce the injectivity
of  $H^{m-1}(G_1;\mathbb Z)\rightarrow H^{m-1}(W\backslash
{\text{Sing}(W)};\mathbb Z)$. On the other hand, from Borel-Moore
homology exact sequence $ 0=H_{m-1}^{BM}({\text{Sing}(W)};\mathbb
Z)\to H_{m-1}(W;\mathbb Z) \to H_{m-1}^{BM}(W\backslash
{\text{Sing}(W)};\mathbb Z)$  we see that $H_{m-1}(W;\mathbb
Z)\subseteq H_{m-1}^{BM}(W\backslash {\text{Sing}(W)};\mathbb
Z)\cong H^{m-1}(W\backslash {\text{Sing}(W)};\mathbb Z)$
(\cite{Fulton2}, p. 217, (26)). So the injective map
$H_{m+1}(G_1;\mathbb Z)\cong H^{m-1}(G_1;\mathbb Z)\rightarrow
H^{m-1}(W\backslash {\text{Sing}(W)};\mathbb Z)$ factors through
$\mu_{m+1}$:
$$
\begin{array}{ccccc}
 H_{m-1}(W;\mathbb Z)&&  \\
\stackrel {\mu_{m+1}}{}\uparrow &\searrow & & \\
 H_{m+1}(G_1;\mathbb Z) & \hookrightarrow & H^{m-1}(W\backslash {\text{Sing}(W)};\mathbb
 Z).
 & &\\
\end{array}
$$
Last claim now follows by Proposition \ref{semi}.
\end{proof}

\begin{theorem}\label{thmb}
Let $X\in |H^0(Y,\mathcal O_Y(d))|$ be a smooth divisor containing
$\Delta$. Then we have:

\smallskip
(a) $H^m(X;\mathbb Z)_{\Delta}$ is freely generated by
$C_1,\dots,C_{\omega}$;

\smallskip
(b) $I_{\Delta}(\mathbb Z)=H^m(Y;\mathbb Z)+H^m(X;\mathbb
Z)_{\Delta}=H^m(Y;\mathbb Z)+H^m(X;\mathbb Z)_{W}$, and
 the monodromy representation on $V_{\Delta}$ for the family of
smooth divisors in $|H^0(Y,\mathcal O_Y(d))|$ containing $\Delta$
is irreducible;

\smallskip
(c) $\dim \left[H^m(Y;\mathbb Q)\cap H^m(X;\mathbb
Q)_{\Delta}\right]=1$.
\end{theorem}

\begin{proof} (a) Since $H_m(\Delta;\mathbb Z)$ is freely generated by
$C_1,\dots,C_{\omega}$ then to prove (a) is equivalent to prove
that the push-forward $H_{m}(\Delta;\mathbb Q)\to H_{m}(X;\mathbb
Q)$ is injective. When $r=1$ this follows by Proposition
\ref{semit}, because in this case $W=\Delta$. Now argue by
induction on $r\geq 2$. Since $\Delta_1=\Delta \cap H_1$, then
$\Delta$ and $\Delta_{1}$ have the same number of components.
Therefore the Gysin map $H_{m}(\Delta;\mathbb Q)\to
H_{m-2}(\Delta_1;\mathbb Q)$ is bijective, and its composition
$\varphi$ with the push-forward $H_{m-2}(\Delta_1;\mathbb Q)\to
H_{m-2}(X_1;\mathbb Q)$ is injective by induction. On the other
hand $\varphi$ is nothing but the composition of the push-forward
$H_{m}(\Delta;\mathbb Q)\to H_{m}(W;\mathbb Q)$ with the Gysin map
$H_{m}(W;\mathbb Q)\to H_{m-2}(X_1;\mathbb Q)$ (observe that
$X_1=W\cap H_1$). Hence the map $H_{m}(\Delta;\mathbb Q)\to
H_{m}(W;\mathbb Q)$ is injective, and so is the map
$H_{m}(\Delta;\mathbb Q)\to H_{m}(X;\mathbb Q)$ by Proposition
\ref{semit} again.

\smallskip
(b) Since $I_{W}(\mathbb Z)\supseteq I_{\Delta}(\mathbb Z)$ and
$I_{\Delta}(\mathbb Z)\supseteq H^m(Y;\mathbb Z)+H^m(X;\mathbb
Z)_{\Delta}$, by   Lemma \ref{gmp} it suffices to prove that
$H^m(X;\mathbb Z)_{W}$ $\subseteq H^m(Y;\mathbb Z)+H^m(X;\mathbb
Z)_{\Delta}$, and that $V_{\Delta}$ is irreducible. So it is
enough to show that for any $0\leq l\leq r-1$ one has
\begin{equation}\label{ninduzione}
H^{m_l}(X_l;\mathbb Z)_{W_l}\subseteq H^{m_l}(Y_l;\mathbb
Z)+H^{m_l}(X_l;\mathbb Z)_{{\Delta}_l}
\end{equation}
($m_l:=m-2l$), and that the monodromy representation on
$V_{\Delta_{l}}$ for the family of smooth divisors $X_l\in
|H^0(Y_l,\mathcal O_{Y_l}(d))|$ containing $\Delta_l$ is
irreducible. To this purpose we argue by decreasing induction on
$l$. When $l=r-1$ we have $\Delta_{r-1}=W_{r-1}$. In this case
(\ref{ninduzione}) is obvious, and the irreducibility of
$V_{\Delta_{r-1}}\,(=I_{W_{r-1}}(\mathbb Q )^{\perp})$ follows
from (\cite{DGF}, Theorem 1.1) because, with the same notation as
in (\cite{DGF}, Theorem 1.1), one has $I_{W_{r-1}}(\mathbb Q
)^{\perp}=H^{m_{r-1}}(X_{r-1};\mathbb Q)_{\perp
W_{r-1}}^{\text{van}}$ (compare with \cite{DGF}, p. 526, from line
27 to 38).

Now assume $0\leq l<r-1$. By induction and Lemma \ref{gmp} (with
$Y_{l+1}\cap G_{l+2}$ instead of $G_1$), we have
\begin{equation}\label{nninduzione}
I_{{W}_{l+1}}(\mathbb Z)=I_{{\Delta}_{l+1}}(\mathbb
Z)=H^{m_{l+1}}(Y_{l+1};\mathbb Z)+H^{m_{l+1}}(X_{l+1};\mathbb
Z)_{{\Delta}_{l+1}}.
\end{equation}
Since $X_{l+1}=W_l\cap H_{l+1}$, then the inclusion map
$i_{X_{l+1}}:X_{l+1}\to W_l$ defines a Gysin map
$i_{X_{l+1}}^{\star}:H_{m_l}(W_l;\mathbb Z)\to
H_{m_{l+1}}(X_{l+1};\mathbb Z)$. Moreover since
$\Delta_{l}\subseteq W_l$, $\Delta_{l+1}=\Delta_{l}\cap H_{l+1}$
(hence the Gysin map $H_{m_l}(\Delta_{l};\mathbb Z)\to
H_{m_{l+1}}(\Delta_{l+1};\mathbb Z)$ is bijective because both
groups are freely generated by the irreducible components), and by
Lefschetz Hyperplane Theorem we have $H^{m_{l+1}}(Y_{l};\mathbb
Z)\cong H^{m_{l+1}}(Y_{l+1};\mathbb Z)$, then
\begin{equation}\label{inclusion}
\Im\,(PD\circ i_{X_{l+1}}^{\star})\supseteq
H^{m_{l+1}}(Y_{l+1};\mathbb Z)+H^{m_{l+1}}(X_{l+1};\mathbb
Z)_{{\Delta}_{l+1}}
\end{equation}
($PD$ means \lq\lq Poincar\'e duality\rq\rq
$H_{m_{l+1}}(X_{l+1};\mathbb Z)\cong H^{m_{l+1}}(X_{l+1};\mathbb
Z)$). By (\cite{DGF}, Lemma 2.3)  $\Im\,(PD\circ
i_{X_{l+1}}^{\star})\otimes_{\mathbb Z} \mathbb Q$ is globally
invariant under the monodromy action on
$H^{m_{l+1}}(X_{l+1};\mathbb Q)$ for the family of smooth divisors
in $|H^0(Y_{l+1},\ic_{Z_{l+1},Y_{l+1}}(d))|$, therefore also for
the family of smooth divisors in
$|H^0(Y_{l+1},\ic_{{\Delta}_{l+1},Y_{l+1}}(d))|$. It follows that
previous inclusion (\ref{inclusion}) is actually an equality:
\begin{equation}\label{xinduzione}
\Im\,(PD\circ i_{X_{l+1}}^{\star})= H^{m_{l+1}}(Y_{l+1};\mathbb
Z)+H^{m_{l+1}}(X_{l+1};\mathbb Z)_{{\Delta}_{l+1}}.
\end{equation}
In fact, by Theorem \ref{finv}, (a), we know that
$H^{m_{l+1}}(X_{l+1};\mathbb Z)\slash I_{{W}_{l+1}}(\mathbb Z)$ is
torsion-free. Therefore if (\ref{inclusion}) would be strict, by
(\ref{nninduzione}) one would have $\left(\Im\,(PD\circ
i_{X_{l+1}}^{\star})\otimes_{\mathbb Z} \mathbb Q\right)\cap
V_{\Delta_{l+1}}\neq \{0\}$, and since $V_{\Delta_{l+1}}$ is
irreducible, it would follow that $H^{m_{l+1}}(X_{l+1};\mathbb
Q)=\Im\,(PD\circ i_{X_{l+1}}^{\star})\otimes_{\mathbb Z} \mathbb
Q$. This is impossible because, for $0\ll \mu_1\ll\dots\ll
\mu_{l+1}$, $h_{m_{l+1}}(X_{l+1};\mathbb Q)$ is arbitrarily large
with respect to $h_{m_l}(W_l;\mathbb Q)$. Again by Theorem
\ref{finv}, (a), (with $R:=W_l$), we know that
$i_{X_{l+1}}^{\star}$ is injective, hence $H_{m_l}(W_{l};\mathbb
Z)\cong \Im\,(PD\circ i_{X_{l+1}}^{\star})$. So from
(\ref{xinduzione}) we get (\ref{ninduzione}), for the  map
$H_{m_l}(W_{l};\mathbb Z)\to H^{m_l}(X_{l};\mathbb Z)$ sends
$H^{m_{l+1}}(X_{l+1};\mathbb Z)_{{\Delta}_{l+1}}$ in
$H^{m_{l}}(X_{l};\mathbb Z)_{{\Delta}_{l}}$ and
$H^{m_{l+1}}(Y_{l+1};\mathbb Z)$ in $H^{m_{l}}(Y_{l};\mathbb Z)$.
Finally note that by Proposition \ref{semi} and (\ref{ninduzione})
it follows $I_{\Delta_l}(\mathbb Q )=I_{W_l}(\mathbb Q )$. So the
irreducibility of $V_{\Delta_{l}}\,(=I_{W_l}(\mathbb Q )^{\perp})$
follows again from (\cite{DGF}, Theorem 1.1).

\smallskip
(c) Again we argue by decreasing induction on $l$. When $l=r-1$
previous equality follows by Proposition \ref{semit} because in
this case $\dim Y_{r-1}=3$ and $\Delta_{r-1}=W_{r-1}$. Now assume
$0\leq l< r-1$. By (a) the Gysin map $H_{m_{l}}(\Delta_{l};\mathbb
Q)\to H_{m_{l+1}}(\Delta_{l+1};\mathbb Q)$ induces an isomorphism
$\varphi: H^{m_{l}}(X_{l};\mathbb Q)_{\Delta_{l}}\to
H^{m_{l+1}}(X_{l+1};\mathbb Q)_{\Delta_{l+1}}$. Suppose that $\xi
\in H^{m_{l}}(Y_{l};\mathbb Q)\cap H^{m_{l}}(X_{l};\mathbb
Q)_{\Delta_{l}}$. Then $\xi \in H^{m_{l}}(Y_{l};\mathbb Q)\cap
H^{m_{l}}(X_{l};\mathbb Q)_{W_{l}}$. By Proposition \ref{semit}
$\xi$ comes from some $\eta\in H^{m_{l+1}}(Y_{l};\mathbb Q)\cong
H^{m_{l+1}}(Y_{l+1};\mathbb Q)$. Therefore pulling-back $\eta$ in
$H^{m_{l+1}}(X_{l+1};\mathbb Q)$ we get $\varphi(\xi)$. This
proves that $\varphi$ sends isomorphically
$H^{m_{l}}(Y_{l};\mathbb Q)\cap H^{m_{l}}(X_{l};\mathbb
Q)_{\Delta_{l}}$ onto $H^{m_{l+1}}(Y_{l+1};\mathbb Q)\cap
H^{m_{l+1}}(X_{l+1};\mathbb Q)_{\Delta_{l+1}}$.
\end{proof}

\medskip
\begin{remark} From (a), (b) (with $\mathbb Q$ instead of $\mathbb Z$)
of previous Theorem \ref{thmb}, and Proposition \ref{semit}, we
deduce that $\dim \left[H^m(Y;\mathbb Q)\cap H^m(X;\mathbb
Q)_{\Delta}\right]=\omega - h^m(W;\mathbb Q)+h^{m-2}(Y;\mathbb
Q)=1$.
\end{remark}

\medskip
We are in position to prove our main result.

\begin{proof}[Proof of Theorem \ref{maintheorem}]
(a) Let $G_1,\dots, G_r$ be general divisors in $|H^0(Y,\mathcal
O_Y(\delta))|$ containing $Z$, and put $\Delta:=X\cap
G_1\cap\dots\cap G_r$. Hence $\Delta$ is a complete intersection
of dimension $r$ containing $Z$, and then (a) follows by (a) of
previous Theorem \ref{thmb}.

\smallskip (b) By (\cite{OS}, 1.2. Theorem) we know that $\Delta\backslash Z$ is
smooth and connected. Observe that $\Delta\neq Z_1\cup\dots\cup
Z_{\rho}$, otherwise $Z=\Delta$ and this is in contrast with the
assumption that $\mathcal I_{Z,Y}(\delta)$ is generated by global
sections. Therefore, apart from $Z_1,\dots ,Z_{\rho}$, $\Delta$
has a unique residual irreducible component $C$, and we have:
\begin{equation}\label{componenti}
\Delta = Z_1 +\dots + Z_{\rho}+C=\delta^r H_X^r\in H^m(X;\mathbb
Z).
\end{equation}
By (\ref{componenti}), and (a) and (c) of Theorem \ref{thmb}, we
deduce $\dim \left[H^m(Y;\mathbb Q)\cap H^m(X;\mathbb
Q)_{Z}\right]$ $=0$. Hence, taking into account that
$H^m(X;\mathbb Z)_Z\subseteq NS_r(X;\mathbb Z)$,  in order to
prove (b) it suffices to prove that
\begin{equation}\label{oldb}
NS_r(X;\mathbb Z)\subseteq  H^m(Y;\mathbb Z)+H^m(X;\mathbb Z)_Z.
\end{equation}
To this aim, first notice that again by (\ref{componenti}) we have
$H^m(Y;\mathbb Z)+H^m(X;\mathbb Z)_{\Delta}=H^m(Y;\mathbb
Z)+H^m(X;\mathbb Z)_{Z}$, which is contained in $I_{Z}(\mathbb
Z)$. Therefore, by (\ref{scatola}), Lemma \ref{gmp} and Theorem
\ref{thmb}, we get
$$
I_{Z}(\mathbb Z)=I_{\Delta}(\mathbb Z)=I_{W}(\mathbb
Z)=H^m(Y;\mathbb Z)+H^m(X;\mathbb Z)_{Z}
$$
($W:=X\cap G_1$). Hence ${I_{Z}(\mathbb Q)}^{\perp}$
($=V_{\Delta}$) is irreducible, and the vanishing cohomology
$H^m(Y;\mathbb Q)^{\perp}$ of $X$ is contained in ${I_{Z}(\mathbb
Q)}^{\perp}+ \,H^{m/2,m/2}(X;\mathbb Q)$. So, as $H^m(Y;\mathbb
Q)^{\perp}$, also ${I_{Z}(\mathbb Q)}^{\perp}$ is not of pure
Hodge type $(m/2,m/2)$. By a standard argument (compare e.g. with
\cite{IJM2}, proof of Theorem 1.1) it follows that $NS_r(X;\mathbb
Q)\subseteq I_{Z}(\mathbb Q)$, from which we get (\ref{oldb})
because $H^{m}(X;\mathbb Z)\slash I_{W}(\mathbb Z)$ is
torsion-free (Theorem \ref{finv}, (a)).

\smallskip (c) In view of (b), to prove (c) is equivalent to prove
that
\begin{equation}\label{equivalent}
NS_r(X;\mathbb Q)\cap H^m(Y;\mathbb Q)=NS_{r+1}(Y;\mathbb Q).
\end{equation}
This is certainly true when $\dim Y=3$ in view of Lefschetz
Theorem on $(1,1)$-Classes (\cite{GH2}, p. 163). So we may assume
$m\geq 4$ and argue by induction on $\dim Y=m+1$. First note that
\begin{equation}\label{ypiuw}
NS_{r}(X;\mathbb Q)= NS_{r+1}(Y;\mathbb Q)+NS_{r}(W;\mathbb Q)
\end{equation}
(observe that $NS_{r}(W;\mathbb Q)\subseteq NS_{r}(X;\mathbb Q)$
by Proposition \ref{semit}). In fact by  Theorem \ref{finv} we
know that $NS_{r+1}(Q;\mathbb Q)$ maps onto $NS_{r}(X;\mathbb Q)$.
On the other hand $NS_{r+1}(P;\mathbb Q)$ maps onto
$NS_{r+1}(Q;\mathbb Q)$ because $P$ is the blowing-up of $Q$ at
the vertex, and moreover we have $NS_{r+1}(P;\mathbb Q)=
NS_{r+1}(Y;\mathbb Q)\oplus NS_{r}(W;\mathbb Q)$ for $P$ is also
the blowing-up of $Y$ at $W$ (\cite{Fulton}, Proposition 6.7, (e), p. 115).
Since $NS_{r}(W;\mathbb Q)=\left[NS_r(W;\mathbb Q)\cap
H^m(Y;\mathbb Q)\right]\oplus H^m(X;\mathbb Z)_{Z}$, by
(\ref{ypiuw}) it follows that in order to prove (\ref{equivalent})
(i.e. (c)) it suffices to prove that
\begin{equation}\label{nypiuw}
NS_r(W;\mathbb Q)\cap H^m(Y;\mathbb Q)=NS_r(W;\mathbb Q)\cap
NS_{r+1}(Y;\mathbb Q).
\end{equation}
To this purpose first observe that by Proposition \ref{semit} we
have:
$$
NS_r(W;\mathbb Q)\cap_{H^m(X;\mathbb Q)} H^m(Y;\mathbb
Q)=NS_r(W;\mathbb Q)\cap_{H^m(X;\mathbb Q)} H^{m-2}(Y;\mathbb Q)
$$
and so
$$
NS_r(W;\mathbb Q)\cap_{H^m(X;\mathbb Q)} H^{m-2}(Y;\mathbb
Q)=NS_r(W;\mathbb Q)\cap_{H_m(W;\mathbb Q)} H^{m-2}(Y;\mathbb Q).
$$
Now consider the following natural commutative diagram:
$$
\begin{array}{ccccc}
H_{m+2}(G_1;\mathbb Q)&\stackrel{}\to & H_m(W;\mathbb Q)\\
\stackrel {}{} \downarrow & & \stackrel {}{}\downarrow &\\
H_{m}(Y_1;\mathbb Q)& \stackrel{}{\to}&
H_{m-2}(X_1;\mathbb Q),\\
\end{array}
$$
where all the maps are Gysin's. It induces the following
commutative diagram:
$$
\begin{array}{ccccc}
NS_{r+1}(G_1;\mathbb Q)&\stackrel{\alpha}\to & NS_r(W;\mathbb Q)\cap_{H_m(W;\mathbb Q)} H^{m-2}(Y;\mathbb Q)\\
\stackrel {\alpha_1}{} \downarrow & & \stackrel {\beta}{}\downarrow &\\
NS_r(Y_1;\mathbb Q)& \stackrel{\beta_1}{\to}&
NS_{r-1}(X_1;\mathbb Q)\cap_{H^{m-2}(X_1;\mathbb Q)} H^{m-2}(Y_1;\mathbb Q).\\
\end{array}
$$
The map $\alpha_1$ is an isomorphism by Theorem \ref{finv}. The
map $\beta_1$ is an isomorphism by induction (compare with
(\ref{equivalent})). The map $\beta$ is an isomorphism by Theorem
\ref{finv} applied to $R:=W$ which only has isolated
singularities. It follows that:
$$
NS_r(W;\mathbb Q)\cap_{H_m(W;\mathbb Q)} H^{m-2}(Y;\mathbb
Q)=NS_{r+1}(G_1;\mathbb Q).
$$
We deduce:
$$
NS_{r+1}(G_1;\mathbb Q)=NS_r(W;\mathbb Q)\cap_{H_m(W;\mathbb Q)}
H^{m-2}(Y;\mathbb Q)
$$
$$
=NS_r(W;\mathbb Q)\cap_{H^m(X;\mathbb Q)} H^{m}(Y;\mathbb
Q)\supseteq NS_r(W;\mathbb Q)\cap_{H^m(X;\mathbb Q)}
NS_{r+1}(Y;\mathbb Q).
$$
But we also have
$$
NS_r(W;\mathbb Q)\cap_{H^m(X;\mathbb Q)} NS_{r+1}(Y;\mathbb
Q)\supseteq NS_{r+1}(G_1;\mathbb Q)
$$
because the cycles of $NS_{r+1}(G_1;\mathbb Q)$  pass through
$NS_{r+1}(Y;\mathbb Q)$ via push-forward $NS_{r+1}(G_1;\mathbb
Q)\subseteq H_{m+2}(G_1;\mathbb Q)\to H_{m+2}(Y;\mathbb Q)$. This
proves (\ref{nypiuw}), hence (c).
\end{proof}

\begin{corollary}
\label{coro1} With the same assumptions as in Theorem
\ref{maintheorem}, suppose also that $\dim Y=3$. Then we have:
$$
Pic(X)= Pic(Y) \oplus H^2(X;\mathbb Z)_Z.
$$
\end{corollary}
\begin{proof} If $\dim Y=3$ then  $m=2$, $r=1$,
and by Lefschetz Theorem on $(1,1)$-Classes we have
$NS_2(Y;\mathbb Z)\subseteq NS_1(X;\mathbb Z)\cap H^2(Y;\mathbb
Z)\subseteq H^{1,1}(Y;\mathbb Z)=NS_2(Y;\mathbb Z)$. So we get $
NS_1(X;\mathbb Z)= NS_2(Y;\mathbb Z) \oplus H^2(X;\mathbb Z)_{Z}$.
And this identification lifts to Picard groups in view of
Lefschetz Hyperplane Theorem and the exponential sequence (compare
with the proof of Corollary 2.3.4, p. 51, in \cite{BS}).
\end{proof}

{\bf{Aknowledgements}}

We would like to thank Ciro Ciliberto and Claudio Murolo for valuable discussions and
suggestions on the subject of this paper.

\end{document}